\documentclass[11pt,reqno]{amsart}

\usepackage{amssymb, amsmath, amsthm}
\usepackage{hyperref}
\usepackage[alphabetic,lite]{amsrefs}
\usepackage{verbatim}
\usepackage{amscd}   
\usepackage[all]{xy} 
\usepackage{youngtab} 
\usepackage{young} 
\usepackage{ytableau}
\usepackage{tikz}
\usepackage{ mathrsfs }
\usepackage{cases}
\usepackage{array}
\usepackage{tabu}

\newcommand{\defi}[1]{{\upshape\sffamily #1}}

\renewcommand{\a}{\alpha}
\renewcommand{\b}{\beta}

\newcommand{\D}{\mathcal{D}}

\newcommand{\bw}{\bigwedge}
\renewcommand{\det}{\textrm{det}}

\renewcommand{\ll}{\lambda}

\newcommand{\oo}{\otimes}

\renewcommand{\P}{\mathcal{P}}

\renewcommand{\t}{\underline{t}}

\newcommand{\x}{\underline{x}}
\newcommand{\y}{\underline{y}}
\newcommand{\z}{\underline{z}}

\newcommand{\X}{\mathcal{X}}

\newcommand{\Z}{\mathcal{Z}}

\newcommand{\Ann}{\operatorname{Ann}}

\newcommand{\Ext}{\operatorname{Ext}}
\newcommand{\GL}{\operatorname{GL}}
\newcommand{\Hom}{\operatorname{Hom}}

\newcommand{\Sym}{\operatorname{Sym}}
\newcommand{\Tor}{\operatorname{Tor}}

\newcommand{\coker}{\operatorname{coker}}
\renewcommand{\det}{\operatorname{det}}

\renewcommand{\ker}{\operatorname{ker}}

\newcommand{\reg}{\operatorname{reg}}

\newcommand{\bb}[1]{\mathbb{#1}}

\renewcommand{\rm}[1]{\textrm{#1}}
\newcommand{\mc}[1]{\mathcal{#1}}
\newcommand{\mf}[1]{\mathfrak{#1}}

\newcommand{\ul}[1]{\underline{#1}}

\def\lra{\longrightarrow}
\def\lla{\longleftarrow}
\def\llra{\longleftrightarrow}

\newtheorem{theorem}{Theorem}[section]
\newtheorem*{theorem*}{Theorem}
\newtheorem*{problem*}{Problem}
\newtheorem{lemma}[theorem]{Lemma}

\newtheorem*{corollary*}{Corollary}

\newtheorem*{main-thm*}{Main Theorem}
\newtheorem*{linear-resolutions*}{Theorem on Linear Resolutions}
\newtheorem*{regularity-powers*}{Theorem on Regularity}
\newtheorem*{injectivity-Ext*}{Theorem on Injectivity of Maps of Ext Modules}
\newtheorem*{Kodaira*}{Kodaira Vanishing for Determinantal Thickenings}

\theoremstyle{definition}
\newtheorem{definition}[theorem]{Definition}
\newtheorem*{definition*}{Definition}
\newtheorem{example}[theorem]{Example}

\newtheorem{problem}[theorem]{Problem}

\theoremstyle{remark}

\newtheorem*{remark*}{Remark}

\numberwithin{equation}{section}

\begin{document}

\title{Homological invariants of determinantal thickenings}

\author{Claudiu Raicu}
\address{Department of Mathematics, University of Notre Dame, 255 Hurley, Notre Dame, IN 46556\newline
\indent Institute of Mathematics ``Simion Stoilow'' of the Romanian Academy}
\email{craicu@nd.edu}


\date{\today}


\begin{abstract} 
 The study of homological invariants such as Tor, Ext and local cohomology modules constitutes an important direction in commutative algebra. Explicit descriptions of these invariants are notoriously difficult to find and often involve combining an array of techniques from other fields of mathematics. In recent years tools from algebraic geometry and representation theory have been successfully employed in order to shed some light on the structure of homological invariants associated with determinantal rings. The goal of this notes is to survey some of these results, focusing on examples in an attempt to clarify some of the more technical statements.
\end{abstract}

\maketitle

\section{Introduction}\label{sec:intro}

Consider a polynomial ring $S = \bb{C}[X_1,\cdots,X_N]$ and a group homomorphism $G\lra \GL_N(\bb{C})$ giving rise to an action of $G$ on $S$ by linear changes of coordinates. It is natural to study the following problem:

\begin{problem}\label{prob:classification}
 Classify the homogeneous ideals $I\subseteq S$ which are invariant under the action of $G$.
\end{problem}

The difficulty of this problem is inversely correlated with the size of the group $G$: when $G=\{1\}$ is the trivial group we get all the ideals in $S$, while for $G=\GL_N(\bb{C})$ the only invariant ideals are the powers $\mf{m}^d$, $d\geq 0$, of the maximal homogeneous ideal $\mf{m} = (X_1,\cdots,X_N)$. An important intermediate case arises by taking $G = (\bb{C}^*)^N$, thought of as the subgroup of diagonal matrices in $\GL_N(\bb{C})$, in which case the invariant ideals are precisely the ones generated by monomials.

In this article we are concerned with the situation when $N = m\cdot n$ for some positive integers $m\geq n$, in which case we write $S = \bb{C}[X_{11},\cdots,X_{mn}]$, we think of the variables as the entries of the generic matrix $X=(X_{ij})$, and we consider $G = \GL_m(\bb{C}) \times \GL_n(\bb{C})$ acting via row and column operations on $X$. In this case Problem~\ref{prob:classification} has been completely resolved in \cite{deconcini-eisenbud-procesi}, and we recall its solution in Section~\ref{sec:class-GL-invariant}.

Once the classification Problem~\ref{prob:classification} is resolved, or perhaps after we restrict our focus to a subclass of $G$-invariant ideals that is of interest, it is natural to study homological invariants associated with these ideals, and understand the symmetries that these invariants inherit from the action of $G$. Some of the most fundamental homological invariants are listed in the following problem.

\begin{problem}\label{prob:invariants}
 Given a $G$-invariant ideal $I\subset S$, describe the $G$-module structure of
 \begin{itemize}
  \item The syzygy modules $\Tor_{\bullet}^S(I,\bb{C})$.
  \item The $\Ext$ modules $\Ext^{\bullet}_S(S/I,S)$.
  \item The local cohomology modules $H_I^{\bullet}(S)$.
 \end{itemize}
\end{problem}

Knowledge of the basic invariants listed above allows us to compute further invariants such as depth and projective dimension, Castelnuovo--Mumford regularity, Hilbert polynomials etc. In the case when $G = \GL_m(\bb{C}) \times \GL_n(\bb{C})$ the $\Ext$ modules have been computed in \cite{raicu-regularity}, while the local cohomology modules have been described in \cite{raicu-weyman}. As far as syzygy modules are concerned, Problem~\ref{prob:invariants} is still unresolved, the only class of examples for which it has been answered being the primary $G$-invariant ideals which are generated by a single irreducible representation of $G$ \cite{raicu-weyman-syzygies}.

Despite the elementary nature of Problems~\ref{prob:classification} and~\ref{prob:invariants}, their resolution depends on some deep and beautiful (both commutative and noncommutative) mathematics. For spaces of matrices the representation theory of the general linear group plays an essential role at every stage. In addition to that, the classification problem is resolved using the theory of algebras with straightening law. Moreover, as far as the homological invariants are concerned: 
\begin{itemize}
 \item The calculation of $\Ext$ modules uses Grothendieck duality, as well as the Borel--Weil--Bott theorem describing the cohomology of simple homogeneous vector bundles on Grassmannian varieties.
 \item The local cohomology modules are best understood through their $\D$-module structure, i.e. their structure as modules over the Weyl algebra of differential operators.
 \item The structure of the syzygy modules depends heavily on the representation theory of the general linear Lie superalgebras.
\end{itemize}

Without going into the details of the tools involved in the proofs, we summarize the existing results along with a number of examples that the reader is encouraged to work out in detail. The organization of the paper is as follows. In Section~\ref{sec:GL_N} we explain the solution to Problems~\ref{prob:classification} and~\ref{prob:invariants} in the case when the group $G$ is the full group of linear automorphisms of $S$, which is equivalent to considering the action by (row and) column operations on matrices with one row, and as such is a special case of the analysis performed in the remaining part of the paper. In Section~\ref{sec:class-GL-invariant} we recall the solution to Problem~\ref{prob:classification} in the case when $G = \GL_m(\bb{C}) \times \GL_n(\bb{C})$ acts on $S=\bb{C}[X_{ij}]$, following \cite{deconcini-eisenbud-procesi}. In Sections~\ref{sec:Ext},~\ref{sec:H_I},~\ref{sec:Tor} we study Problem~\ref{prob:invariants} for each of the three fundamental homological invariants, following \cites{raicu-regularity,raicu-weyman,raicu-weyman-witt,raicu-weyman-syzygies}. We end with Section~\ref{sec:open} where we give a short list of open problems.

\section{Ideals invariant under all coordinate changes}\label{sec:GL_N}

Let $S=\bb{C}[X_1,\cdots,X_N]$ and consider the action of $G=\GL_N(\bb{C})$ by linear changes of coordinates. If we regard $X_1,\cdots,X_N$ as the entries of the generic $N\times 1$ matrix, then the results of this section are special cases ($m=N$ and $n=1$) of the results in the rest of the paper. They should be regarded as a warm-up for things to come. If we write
\[\mf{m} = \langle X_1,\cdots,X_N\rangle\]
for the maximal homogeneous ideal of $S$ then the answer to Problem~\ref{prob:classification} is given by the following.

\begin{theorem}\label{thm:invariant-GLN}
 The $\GL_N(\bb{C})$-invariant ideals in $S$ are $\mf{m}^d$, for $d\geq 0$.
\end{theorem}

\begin{proof}
 The polynomial ring $S$ decomposes as a direct sum of irreducible $\GL_N(\bb{C})$-representations 
 \[ S = \bigoplus_{d\geq 0} \Sym^d(\bb{C}^N),\]
 and $\mf{m}^d$ is precisely the ideal generated by $\Sym^d(\bb{C}^N)$, which is therefore invariant. Any $\GL_N(\bb{C})$-invariant ideal is generated by a (finite dimensional) subrepresentation of $S$, which necessarily has the form
 \[ \Sym^{d_1}(\bb{C}^N) \oplus \cdots \oplus \Sym^{d_r}(\bb{C}^N).\]
 This means that $I = \mf{m}^{d_1} + \cdots + \mf{m}^{d_r}$, which implies $I = \mf{m}^d$ for $d=\min\{d_1,\cdots,d_r\}$.
\end{proof}

Having classified the invariant ideals, we now turn to the homological invariants from Problem~\ref{prob:invariants}. It is useful to note (see \cite[Exercise~A.2.17]{eisenbud-CA}) that $\mf{m}^d$ ($d>0$) is the ideal generated by the maximal minors of the $(d+N-1)\times d$ matrix
\[
\begin{bmatrix}
 X_1 & X_2 & X_3 & \cdots & X_N & 0 & 0 & \cdots & 0 \\
 0 & X_1 & X_2 & \cdots & X_{N-1} & X_N & 0 & \cdots & 0 \\
 & & \ddots & \ddots & & & \ddots & &  \\
 & & & \ddots & \ddots & & & \ddots & 0 \\
 0 & \cdots & \cdots & 0 & X_1 & X_2 & \cdots & \cdots & X_N
\end{bmatrix}
\]
Since the ideal $\mf{m}^d$ has grade $N = (N+d-1) - d + 1$, we can apply \cite[Theorem~A.2.60]{eisenbud-syzygies} to conclude that it has a linear resolution given by the Eagon--Northcott complex. To describe the $\GL_N(\bb{C})$-structure of the syzygy modules, we first introduce some notation.

A \defi{partition} $\ul{x}=(x_1,x_2,\cdots)$ is a finite collection of non-negative integers, with $x_1\geq x_2\geq\cdots$. We call each $x_i$ a \defi{part} of $\ul{x}$, and define the \defi{size} of $\ul{x}$ to be $|\ul{x}|=x_1+x_2+\cdots$. Most of the time we suppress the parts of size zero from the notation, for instance the partitions $(4,2,1,0,0)$ and $(4,2,1)$ are considered to be the same; their size is $7=4+2+1$. When $\x$ has repeated parts, we often use the abbreviation $(b^a)$ for the sequence $(b,b,\cdots,b)$ of length $a$. For instance, $(4,4,4,3,3,3,3,3,2,1)$ would be abbreviated as $(4^3,3^5,2,1)$. We denote by $\mc{P}_n$ the collection of partitions with at most $n$ non-zero parts. It is often convenient to identify a partition $\x$ with the associated Young diagram:
\[\Yvcentermath1 \x = (4,2,1,0,0) \quad\quad\llra\quad\quad \yng(4,2,1)\]
Of particular interest in this section are the \defi{hook partitions}, i.e. those $\x$ for which $x_2\leq 1$, or equivalently $\x = (a,1^b)$ for some $a,b\geq 0$. The name is illustrative of the shape of the associated Young diagram, as seen for instance in the following example of a hook partition:
\[\Yvcentermath1 \x = (4,1,1) \quad\quad\llra\quad\quad \yng(4,1,1)\]

Each partition $\x\in\mc{P}_N$ determines an irreducible $\GL_N(\bb{C})$-representation, denoted $\bb{S}_{\x}\bb{C}^N$ (one usually refers to $\bb{S}_{\x}$ as the \defi{Schur functor} associated to the partition $\x$): our conventions are such that when $\x=(d)$ we have that $\bb{S}_{\x}\bb{C}^N=\Sym^d\bb{C}^N$ is a symmetric power, while for $\x=(1^k)$ we have that $\bb{S}_{\x}\bb{C}^N=\bw^k\bb{C}^N$ is an exterior power. The syzygies of powers of the maximal homogeneous ideals can then be described in terms of hook partitions as follows (unless specified, all tensor products are considered over the base field $\bb{C}$).

\begin{theorem}[{\cite[(1.a.10)]{green2}, \cite[Cor.~3.2]{buchsbaum-eisenbud}}]\label{thm:syz-GLN}
 We have $\GL_N(\bb{C})$-equivariant isomorphisms
 \[\Tor_p^S(\mf{m}^d,\bb{C})_{d+p} \simeq \bb{S}_{d,1^p}\bb{C}^N,\mbox{ for }p=0,\cdots,N-1,\]
 and $\Tor_p^S(\mf{m}^d,\bb{C})_q = 0$ for all other values of $p,q$. Equivalently, $\mf{m}^d$ has a linear minimal free resolution
 \[0\lla \mf{m}^d \lla \bb{S}_{d}\bb{C}^N\oo S(-d) \lla \cdots \lla \bb{S}_{d,1^p}\bb{C}^N \oo S(-d-p) \lla \cdots \lla \bb{S}_{d,1^{N-1}}\bb{C}^N \oo S(-d-N+1) \lla 0\]
\end{theorem}

In order to describe the graded components of $\Ext$ and local cohomology modules, we need to consider a generalization of partitions where we allow some parts to be negative: we define a \defi{dominant weight} to be an element $\ll\in\bb{Z}^N$ with $\ll_1\geq\cdots\geq\ll_N$. The associated Schur functor $\bb{S}_{\ll}$ has the property that
\[ \bb{S}_{\ll + (1^N)} \bb{C}^N \simeq \bb{S}_{\ll} \bb{C}^N \oo \bw^N \bb{C}^N\mbox{ for all dominant weights }\ll,\]
and in particular
\[ \bb{S}_{(-1^N)} \bb{C}^N\simeq \left(\bw^N \bb{C}^N\right)^{\vee} = \Hom_{\bb{C}}\left(\bw^N \bb{C}^N,\bb{C}\right).\]
More generally, if we let $\ll^{\vee} = (-\ll_N,-\ll_{N-1},\cdots,-\ll_1)$ then we obtain a $\GL_N(\bb{C})$-equivariant isomorphism
\begin{equation}\label{eq:lambda-dual}
 \bb{S}_{\ll^{\vee}}\bb{C}^N \simeq \Hom_{\bb{C}}\left(\bb{S}_{\ll}\bb{C}^N,\bb{C}\right).
\end{equation}
With the above conventions, the $\GL_N(\bb{C})$-equivariant structure on $\Ext$ modules is given by the following.

\begin{theorem}\label{thm:Ext-GLN}
 For $d\geq 0$ we have a graded $\GL_N(\bb{C})$-equivariant isomorphism
 \begin{equation}\label{eq:Ext-GLN}
 \Ext^N_S(S/\mf{m}^d,S) \simeq \bigoplus_{i=0}^{d-1} \bb{S}_{(-1^{N-1},-i-1)}\bb{C}^N,
 \end{equation}
 where $\bb{S}_{(-1^{N-1},-i-1)}\bb{C}^N$ is placed in degree $-i-N$. Moreover, for $j\neq N$ we have $\Ext^j_S(S/\mf{m}^d,S) = 0$.
\end{theorem}

\begin{proof}
 We note that $A = S/\mf{m}^d$ is a graded Artinian $\bb{C}$-algebra which is a quotient of the polynomial ring, so duality theory (see \cite[Chapter~21]{eisenbud-CA} or \cite[Section~I.3]{bruns-herzog}) provides canonical isomorphisms
 \[ \Hom_{\bb{C}}(A,\bb{C}) \simeq \omega_A \simeq \Ext^N_S(A,\omega_S),\]
 where $\omega$ denotes the canonical module. Using the fact that
 \[\omega_S \simeq \bb{S}_{(1^N)}\bb{C}^N \oo S(-N)\]
 and $A = \bigoplus_{i=0}^{d-1} \bb{S}_{i} \bb{C}^N$, which in turn implies using (\ref{eq:lambda-dual})
 \[\Hom_{\bb{C}}(A,\bb{C}) \simeq \bigoplus_{i=0}^{d-1} \bb{S}_{(0^{N-1},-i)} \bb{C}^N,\]
 we obtain
 \[ \Ext^N_S(A,S) \simeq \Hom_{\bb{C}}(A,\bb{C}) \oo \bb{S}_{(-1^N)}\bb{C}^N \simeq \bigoplus_{i=0}^{d-1} \bb{S}_{(-1^{N-1},-i-1)}\bb{C}^N\]
 where $\bb{S}_{(-1^N)}\bb{C}^N$ is placed in degree $-N$, which determines the desired grading on the remaining homogeneous components.
\end{proof}

Using the fact that local cohomology is a direct limit of $\Ext$ modules (see \cite[Appendix~4]{eisenbud-CA})
\begin{equation}\label{eq:loccoh=limExt}
 H_I^j(S) = \varinjlim_d \Ext^j_S(S/I^d,S)
\end{equation}
one easily obtains the following.

\begin{theorem}\label{thm:loccoh-GLN}
We have a graded $\GL_N(\bb{C})$-equivariant isomorphism
 \begin{equation}\label{eq:loccoh-GLN}
 H_{\mf{m}}^N(S) \simeq \bigoplus_{i\geq 0} \bb{S}_{(-1^{N-1},-i-1)}\bb{C}^N,
 \end{equation}
 where $\bb{S}_{(-1^{N-1},-i-1)}\bb{C}^N$ is placed in degree $-i-N$. Moreover, for $j\neq N$ we have $H_{\mf{m}}^j(S) = 0$.
\end{theorem}

\begin{proof} The formula (\ref{eq:loccoh-GLN}) follows from Theorem~\ref{thm:Ext-GLN} and (\ref{eq:loccoh=limExt}) once we show that the maps in the directed system from (\ref{eq:loccoh=limExt}) are injections. Using the long exact sequence obtained by applying $\Hom_S(\bullet,S)$ to the short exact sequence
\[ 0 \lra \mf{m}^d/\mf{m}^{d+1} \lra S/\mf{m}^{d+1} \lra S/\mf{m}^d \lra 0\]
we see that it is sufficient to verify that $\Ext^{N-1}_S(\mf{m}^d/\mf{m}^{d+1},S) = 0$. Since $\mf{m}^d/\mf{m}^{d+1}$ is annihilated by $\mf{m}$ we get in fact that $\Ext^j_S(\mf{m}^d/\mf{m}^{d+1},S) = 0$ for all $j<N$, concluding the proof.
\end{proof}

An alternative way of computing local cohomology is using the \v Cech complex, which yields the more familiar isomorphism
\[H_{\mf{m}}^N(S) \simeq \frac{ S_{X_1\cdots X_N} } {\sum_{i=1}^N S_{X_1\cdots\hat{X_i}\cdots X_N}}.\]
Nevertheless, the $\GL_N(\bb{C})$-action is harder to see from this description, since the \v Cech complex is not $\GL_N(\bb{C})$-equivariant. The advantage of this representation is that it manifestly expresses $H_{\mf{m}}^N(S)$ as a module over the \defi{Weyl algebra} $\D = S\langle \partial_1,\cdots,\partial_N\rangle$, where $\partial_i = \partial/\partial X_i$ is the partial derivative with respect to~$X_i$. Even more remarkable, the $\D$-module $H_{\mf{m}}^N(S)$ is simple, i.e. it has no non-trivial submodules, despite the fact that as an $S$-module it is not even finitely generated! Other ways of regarding $H_{\mf{m}}^N(S)$ are as the injective hull of the residue field $\bb{C}=S/\mf{m}$, or as the graded dual of the polynomial ring $S$, but they will not concern us any further.

\section{The classification of $\GL$-invariant ideals}\label{sec:class-GL-invariant}

Consider positive integers $m\geq n\geq 1$ and let $S=\bb{C}[X_{ij}]$, where $i=1,\cdots,m$ and $j=1,\cdots,n$. We identify $S$ with the symmetric algebra $\Sym(\bb{C}^m \oo \bb{C}^n)$ and consider the group $\GL=\GL_m(\bb{C}) \times \GL_n(\bb{C})$ together with its natural action on $S$. The goal of this section is to recall from \cite{deconcini-eisenbud-procesi} the classification of $\GL$-invariant ideals $I\subseteq S$, which will be used throughout the rest of the paper.

Given our notation and conventions regarding partitions and the associated Young diagrams (see Section~\ref{sec:GL_N}), whenever we refer to a row/column of a partition $\x$, we always mean a row/column of the associated Young diagram. Given a partition $\x$, we can then construct the conjugate partition $\x'$ by transposing the associated Young diagrams: $x'_i$ counts the number of boxes in the $i$-th column of $\x$, e.g. $(4,2,1)'=(3,2,1,1)$.

We define for each $l=1,\cdots,n$ the polynomial 
\begin{equation}\label{eq:detl}
\det_l = \det(X_{ij})_{1\leq i,j\leq l}.
\end{equation}
For $\x\in\P_n$ we define
\begin{equation}\label{eq:detx}
 \det_{\x} = \prod_{i=1}^{x_1} \det_{x'_i},
\end{equation}
and let
\begin{equation}\label{eq:defIx}
I_{\x} = \langle \GL\cdot\det_{\x}\rangle
\end{equation}
be the ideal generated by the $\GL$-orbit of the polynomial $\det_{\x}$. More generally, if $\mc{X}\subseteq \P_n$ we let
\begin{equation}\label{eq:defIX}
 I_{\mc{X}}=\sum_{\ul{x}\in\mc{X}} I_{\ul{x}}.
\end{equation}

\begin{example}
 Consider $\x = (1^p)$ for some $p\leq n$. Since $\x$ has a single column, it follows from (\ref{eq:detx}) that
 \[\det_{\x} = \det_p = \det(X_{ij})_{1\leq i,j\leq p}.\]
 Using the multilinear property of the determinant, it is easy to check that the $\bb{C}$-span of $\GL\cdot\det_{\x}$ is the same as that of the $p\times p$ minors of the generic matrix $(X_{ij})$. We write $I_p$ instead of $I_{(1^p)}$ for the ideal generated by these minors, and recall that $I_p$ is a prime ideal corresponding to the affine algebraic variety of matrices of rank $<p$.
\end{example}

\begin{example}
 When $n=1$ (and $m=N$) we have that every $\x\in\P_1$ has the form $\x=(d)$ for some non-negative integer $d$. If we write $X_i$ instead of $X_{i1}$ then we get that $S = \bb{C}[X_1,\cdots,X_N]$, $\det_1 = X_1$ and
 \[\det_{\x} = X_1^d.\]
One can show that the $\bb{C}$-span of $\GL\cdot X_1^d$ coincides with the vector space of homogeneous polynomials of degree $d$, i.e. $I_{\x} = \mf{m}^d$ using the notation from Section~\ref{sec:GL_N}. The acute reader might have noticed that here we are working with the product of groups $\GL_N(\bb{C}) \times \GL_1(\bb{C})$ instead of simply $\GL_N(\bb{C})$, but in fact the action of $\GL_1(\bb{C})$ is subsumed by that of $\GL_N(\bb{C})$, so the case $n=1$ is precisely the one studied in Section~\ref{sec:GL_N}.
\end{example}

Given two partitions $\x,\y\in\P_n$, we write $\x\leq\y$ if $x_i\leq y_i$ for all $i$. We say that $\x$ and $\y$ are \defi{incomparable} if neither $\x\leq\y$ nor $\y\leq x$. It is shown in \cite{deconcini-eisenbud-procesi} that
\begin{equation}\label{eq:incl-Ix-Iy}
 \x\leq\y \Longleftrightarrow I_{\x} \supseteq I_{\y},
\end{equation}
which in the case when $n=1$ is simply the statement that $d\leq e$ if and only if $\mf{m}^d \supseteq \mf{m}^e$. It follows that when considering ideals of the form (\ref{eq:defIX}) there is no harm in assuming that the partitions in $\X$ are incomparable. The following theorem completely answers Problem~\ref{prob:classification} for $\GL$-invariant ideals.

\begin{theorem}[\cite{deconcini-eisenbud-procesi}]\label{thm:GL-classification}
 The association $\mc{X}\lra I_{\mc{X}}$ establishes a bijective correspondence between
\[ \{\mbox{(finite) subsets }\mc{X}\subset \P_n\mbox{ consisting of incomparable partitions}\} \llra \{\GL-\mbox{invariant ideals }I\subseteq S\}.\]
\end{theorem}

Perhaps the most interesting examples of $\GL$-invariant ideals are the powers (usual, symbolic, or saturated) of the determinantal ideals $I_p$. Each such power corresponds via Theorem~\ref{thm:GL-classification} to a finite set of partitions, which can be described as follows. For the usual powers, it is shown in \cite{deconcini-eisenbud-procesi} that $I_p^d = I_{\X_p^d}$ where
\begin{equation}\label{eq:defXpd}
\mc{X}_p^d = \{\x\in\P_n: |\ul{x}|=p\cdot d,\ x_1\leq d\}
\end{equation}

\begin{example}\label{example:X-for-powers}
 Suppose that $m=n=3$ (or that $m\geq n=3$) and let $p=2$. The following table records the partitions corresponding to small powers of the ideal of $2\times 2$ minors.
\begin{center}
\setlength{\extrarowheight}{2pt}
\ytableausetup{smalltableaux,aligntableaux=center}
\tabulinesep=1.2mm
\begin{tabu}{c|c}
$d$ & $\X_2^d$ \\
\hline
1 & \ydiagram{1,1}\\
\hline
2 & \ydiagram{2,2}\quad \ydiagram{2,1,1} \\
\hline
3 & \ydiagram{3,3}\quad \ydiagram{3,2,1}\quad \ydiagram{2,2,2} \\
\hline
4 & \ydiagram{4,4}\quad \ydiagram{4,3,1}\quad \ydiagram{4,2,2} \quad \ydiagram{3,3,2}\\
\hline
5 & \ydiagram{5,5}\quad \ydiagram{5,4,1}\quad \ydiagram{5,3,2} \quad \ydiagram{4,4,2} \quad \ydiagram{4,3,3}\\
\end{tabu}
\end{center}
\end{example}

The symbolic and saturated powers of ideals of minors are obtained via saturation as follows. Recall that the \defi{saturation} of an ideal $I$ with respect to $J$ is defined via
\begin{equation}\label{eq:def-saturation}
I:J^{\infty} = \{f\in S: f\cdot J^r \subseteq I\textrm{ for }r\gg 0\}.
\end{equation}
If we let $I_p^{(d)}$ denote the $d$-th symbolic power of $I_p$, and let $(I_p^d)^{sat}$ denote the saturation of $I_p^d$ with respect to the maximal homogeneous ideal, then we have
\[ I_p^{(d)} = I_p^d : I_{p-1}^{\infty}\quad\mbox{ and }\quad\mbox (I_p^d)^{sat} = I_p^d : I_1^{\infty}.\]
It is therefore important to understand how the ideals $I_{\X}$ transform under saturation. To do so, we need to introduce one more piece of notation: given a positive integer $c$, we write $\x(c)$ for the partition defined by $\x(c)_i=\min(x_i,c)$, so that the non-zero columns of $\x(c)$ are precisely the first $c$ columns of $\x$. We then have the following.

\begin{lemma}[{\cite[Lemma~2.3]{raicu-regularity}}]\label{lem:saturation} 
 For a subset $\X\subset\P_n$ we have $I_{\X} : I_p^{\infty} = I_{\X^{:p}}$ where
\begin{equation}\label{eq:p-saturation}
 \X^{:p} = \{\x(c): \x\in\X,c\in\bb{Z}_{\geq 0}, x'_c> p\rm{ if }c>0,\rm{ and }x'_{c+1}\leq p\}
\end{equation}
\end{lemma}
Thinking more concretely in terms of the corresponding Young diagrams, (\ref{eq:p-saturation}) asserts that the partitions in $\X^{:p}$ are obtained from those of $\X$ by eliminating columns of size $\leq p$. Based on this observation, one can show that $I_p^{(d)} = I_{\X_p^{(d)}}$ where
\begin{equation}\label{eq:def-X_p^(d)}
\X_p^{(d)} = \{\x\in\P_n: x_1=\cdots=x_p, x_p + x_{p+1} + \cdots + x_n = d\}
\end{equation}
Furthermore, we have that $(I_p^d)^{sat} = I_{(\X_p^d)^{sat}}$ where
\begin{equation}\label{eq:saturated-pows}
 (\mc{X}_p^d)^{sat} = \{\x\in\P_n: x_1=x_2, x_2 + x_3 + \cdots + x_n = (p-1)\cdot d\}
\end{equation}
It is important to note that both $\X_p^{(d)}$ and $(\mc{X}_p^d)^{sat}$ consist of incomparable partitions, whereas the set $\X^{:p}$ defined in (\ref{eq:p-saturation}) may contain comparable partitions even if $\X$ does not (see the example below)!

\begin{example}\label{example:X-for-symbsat-powers}
 Continuing with the notation in Example~\ref{example:X-for-powers} we get the following table recording the partitions corresponding to small symbolic powers of the ideal of $2\times 2$ minors (note that for $2\times 2$ minors the symbolic powers coincide with the saturated powers).
\begin{center}
\setlength{\extrarowheight}{2pt}
\ytableausetup{smalltableaux,aligntableaux=center}
\tabulinesep=1.2mm
\begin{tabu}{c|c|c}
$d$ & $(\X_2^d)^{:1}$ & $\X_2^{(d)} = (\X_2^d)^{sat}$ \\
\hline
1 & \ydiagram{1,1} & \ydiagram{1,1}\\
\hline
2 & \ydiagram{2,2}\quad \ydiagram{1,1,1} & \ydiagram{2,2}\quad \ydiagram{1,1,1} \\
\hline
3 & \ydiagram{3,3}\quad \ydiagram{2,2,1}\quad \ydiagram{2,2,2} & \ydiagram{3,3}\quad \ydiagram{2,2,1} \\
\hline
4 & \ydiagram{4,4}\quad \ydiagram{3,3,1}\quad \ydiagram{2,2,2} \quad \ydiagram{3,3,2} & \ydiagram{4,4}\quad \ydiagram{3,3,1}\quad \ydiagram{2,2,2} \\
\hline
5 & \ydiagram{5,5}\quad \ydiagram{4,4,1}\quad \ydiagram{3,3,2} \quad \ydiagram{4,4,2} \quad \ydiagram{3,3,3} & \ydiagram{5,5}\quad \ydiagram{4,4,1}\quad \ydiagram{3,3,2} \\
\end{tabu}
\end{center}
To put into words some of the information in this table, consider for instance $\X_2^{(2)} = \{(1,1,1),(2,2)\}$. Recall that $I_2$ is the defining ideal of the \defi{Segre variety} $Z$ consisting of matrices of rank at most $1$, and that $I_2^{(2)}$ is the ideal of functions vanishing to order at least two along $Z$ (see \cite[Section~3.9]{eisenbud-CA}). The fact that $I_2^{(2)} = I_{\X_2^{(2)}}$ can then be interpreted as follows:
\begin{itemize}
 \item The function $f_1 = \det_3 = \det(X_{ij})_{1\leq i,j\leq 3}$ vanishes to order two along $Z$, as well as every other function in its $\GL$-orbit.
 \item The function $f_2 = \det_2^2 = \left(\det(X_{ij})_{1\leq i,j\leq 2}\right)^2$ vanishes to order two along $Z$, as well as every other function in its $\GL$-orbit.
 \item Every function vanishing to order two along $Z$ is contained in the ideal generated by the $\GL$-orbits of $f_1$ and $f_2$.
\end{itemize}
\end{example}

\begin{example}\label{example:X_3^3}
 For an example where usual, symbolic and saturated powers are all distinct, consider the case when $m=n=4$, $p=3$, and $d=3$. We have
 \[\X_3^3 = \left\{ \ydiagram{3,3,3},\ \ydiagram{3,3,2,1},\ \ydiagram{3,2,2,2} \right\} \]
 \[(\X_3^3)^{sat} = \left\{ \ydiagram{3,3,3},\ \ydiagram{3,3,2,1},\ \ydiagram{2,2,2,2} \right\} \]
 \[\X_3^{(3)} = \left\{ \ydiagram{3,3,3},\ \ydiagram{2,2,2,1} \right\} \]
and note that (based on (\ref{eq:incl-Ix-Iy})) we get strict inclusions $I_3^3 \subsetneq (I_3^3)^{sat} \subsetneq I_3^{(3)}$.
\end{example}

\section{$\Ext$ modules}\label{sec:Ext}

Having described the solution to the classification Problem~\ref{prob:classification} for $\GL$-invariant ideals in $S=\bb{C}[X_{ij}]$, we now turn to the analysis of the basic homological invariants. In this section we explain the solution to Problem~\ref{prob:invariants} for the $\Ext$ groups following \cite{raicu-regularity}. The main result asserts that for the purpose of computing $\Ext^{\bullet}_S(S/I,S)$ it is enough to replace $S/I$ with the associated graded $\mf{gr}(S/I)$ for a natural finite filtration of $S/I$ where the quotients have explicitly computable $\Ext$ modules. In more technical terms, the main result is about the degeneration of the spectral sequence for computing $\Ext$, associated to the said filtration of $S/I$, but as far as we understand it this degeneration occurs for highly non-trivial reasons. For instance the proofs that we give in \cite{raicu-regularity} require an analysis of all (or at least a large class) of $\GL$-invariant ideals, and would not be applicable on a case to case basis: for example we do not know how to compute directly $\Ext^{\bullet}_S(S/I,S)$ when $I=I_p^d$ is a power of a determinantal ideal, without doing so for arbitrary (or at least sufficiently general) $\GL$-invariant ideals $I$.

\begin{theorem}\label{thm:main-Ext}
 To any $\GL$-invariant ideal $I\subseteq S$ we can associate a finite set $\mc{M}(I)$ of $\GL$-equivariant $S$-modules, arising as successive quotients in a natural filtration of $S/I$, and having the property that for each $j\geq 0$ we have a $\GL$-equivariant degree preserving isomorphism (but not an $S$-module isomorphism!)
 \[\Ext^j_S(S/I,S) \simeq \bigoplus_{M\in\mc{M}(I)} \Ext^j_S(M,S),\]
The sets $\mc{M}(I)$ and the modules $\Ext^j_S(M,S)$ for $M\in\mc{M}(I)$ can be computed explicitly. Furthermore, the association $I\mapsto\mc{M}(I)$ has the property that whenever  $I\supseteq J$ are $\GL$-invariant ideals, the (co)kernels and images of the induced maps $\Ext^j_S(S/I,S)\lra\Ext^j_S(S/J,S)$ can be computed as follows.
  \[\ker\left(\Ext^j_S(S/I,S)\lra\Ext^j_S(S/J,S)\right) = \bigoplus_{M\in\mc{M}(I)\setminus\mc{M}(J)}\Ext^{j}_S(M,S),\]
  \[\operatorname{Im}\left(\Ext^j_S(S/I,S)\lra\Ext^j_S(S/J,S)\right) = \bigoplus_{M\in\mc{M}(I)\cap\mc{M}(J)}\Ext^{j}_S(M,S),\]
  \[\coker\left(\Ext^j_S(S/I,S)\lra\Ext^j_S(S/J,S)\right) = \bigoplus_{M\in\mc{M}(J)\setminus\mc{M}(I)}\Ext^{j}_S(M,S).\]
 Finally, the $\Ext$ modules get smaller under saturation in a very precise sense: we have that
 \begin{equation}\label{eq:M-saturation}
 \mc{M}(I:I_p^{\infty}) = \{M\in\mc{M}(I): \Ann(M) \not\subseteq I_p\} 
 \end{equation}
\end{theorem}

The equivariant modules $M$ appearing in the sets $\mc{M}(I)$ are indexed by pairs $(\z,l)$ where $\z\in\P_n$ is a partition and $l$ is a non-negative integer (we write $M=J_{\z,l}$ for the module corresponding to the pair $(\z,l)$). We think of $\z$ as combinatorial data, and of $l$ as geometric data: $l$ indicates the fact that the scheme theoretic support of $J_{\z,l}$ is precisely the variety of matrices of rank $\leq l$, or equivalently it says that the annihilator of $J_{\z,l}$ is the ideal $I_{l+1}$. To define $J_{\z,l}$ as a quotient of ideals we consider
\begin{equation}\label{eq:defSucc}
 \mf{succ}(\z,l) = \{\x\in\P_n:\x\geq\z\rm{ and }x_i>z_i\rm{ for some }i>l\}.
\end{equation}
and define (using notation (\ref{eq:defIx}) and (\ref{eq:defIX}))
\begin{equation}\label{eq:defJzl}
 J_{\z,l} = I_{\z}/I_{\mf{succ}(\z,l)}.
\end{equation}

\begin{example}\label{ex:Jzl-z=0}
 When $\z=(0)$ is the empty partition we get that $I_{\z} = S$ and $I_{\mf{succ}(\z,l)} = I_{l+1}$ so that
 \[ J_{(0),l} = S/I_{l+1}.\]
\end{example}

\begin{example}\label{ex:Jzl-l=0}
 When $l=0$ we have that $I_{\mf{succ}(\z,0)} = I_{\z} \cdot I_1$ and $J_{\z,0}$ is identified with the vector space of minimal generators of $I_{\z}$ (since $I_1=\mf{m}$ is the maximal homogeneous ideal of $S$). The vector space of minimal generators of $I_{\z}$ is the $\bb{C}$-span of the $\GL$-orbit $\GL\cdot\det_{\z}$ and it is isomorphic to the irreducible $\GL$-representation $\bb{S}_{\z}\bb{C}^m \oo \bb{S}_{\z}\bb{C}^n$. If we further assume that $n=1$ (and $m=N$) then $\z = (d)$ for some $d\geq 0$ and $J_{\z,0} = \mf{m}^d/\mf{m}^{d+1}\simeq \Sym^d\bb{C}^N$.
\end{example}

To explain which of the modules $M = J_{\z,l}$ appear in $\mc{M}(I)$ we need the following key definition. Recall that $\x'$ denotes the conjugate partition to $\x$, and that $\x(c)$ is the partition obtained from the first $c$ columns of $\x$, namely $\x(c)_i = \min(x_i,c)$ for all $i$.

\begin{definition}\label{def:ZX}
 For $\mc{X}\subset\P_n$ a finite subset we define $\mc{Z}(\mc{X})$ to be the set consisting of pairs $(\z,l)$ where $\z\in\P_n$ and $l\geq 0$ are such that if we write $c=z_1$ then the following hold:
 \begin{enumerate}
  \item There exists a partition $\ul{x}\in\mc{X}$ such that $\ul{x}(c)\leq\z$ and $x'_{c+1}\leq l+1$.
  \item For every partition $\ul{x}\in\mc{X}$ satisfying (1) we have $x'_{c+1}=l+1$.
 \end{enumerate}
\end{definition}

With this definition, we can make explicit the sets $\mc{M}(I)$ in Theorem~\ref{thm:main-Ext}: if $I = I_{\X}$ for $\X\subset\P_n$ then
\begin{equation}\label{eq:MIX}
 \mc{M}(I_{\X}) = \{ J_{\z,l} : (\z,l) \in \Z(\X)\}.
\end{equation}
Notice that Definition~\ref{def:ZX} does not require that the set $\X$ consist of incomparable partitions, so implicit in equation (\ref{eq:MIX}) is the fact that $\Z(\X) = \Z(\X')$ whenever $I_{\X}=I_{\X'}$, i.e. whenever $\X$ and $\X'$ have the same set of minimal partitions (with respect to $\leq$).

\begin{example}\label{ex:MI_p}
 If $\X = \{ (1^{l+1}) \}$ (so that $I_{\X} = I_{l+1}$) then one can check using Definition~\ref{def:ZX} that
 \[ \Z(\X) = \{ ((0),l) \}\]
 and therefore $\mc{M}(I_{l+1}) = \{ S/I_{l+1} \}$ consists of a single module.
\end{example}

\begin{example}\label{ex:MI-n=1}
 Suppose that $n=1$ and let $\X = \{ (d) \}$, so that $I_{\X} = \mf{m}^d$. We have that
 \[ \Z(\X) = \{ ((i),0) : i=0,\cdots,d-1 \} \]
 and using the calculation from Example~\ref{ex:Jzl-l=0} we get
 \[ \mc{M}(\mf{m}^d) = \{ \mf{m}^i/\mf{m}^{i+1} : i=0,\cdots,d-1\}.\]
 The conclusions of Theorem~\ref{thm:main-Ext} are then easy to verify in this situation (see also Theorem~\ref{thm:Ext-GLN}).
\end{example}

Just as we did in Section~\ref{sec:class-GL-invariant}, we would like to understand better the powers (usual, symbolic, saturated) of the determinantal ideals. To do so we need to describe the sets $\Z(\X_p^d)$, $\Z(\X_p^{(d)})$, and $\Z((\X_p^d)^{sat})$. In view of Lemma~\ref{lem:saturation} and (\ref{eq:MIX}), we can rewrite (\ref{eq:M-saturation}) more explicitly as
\begin{equation}\label{eq:Z-saturation}
\Z(\X^{:p}) = \{(\z,l)\in\Z(\X): l\geq p\} \subseteq \Z(\X),
\end{equation}
so knowing $\Z(\X_p^d)$ immediately determines $\Z(\X_p^{(d)})$ and $\Z((\X_p^d)^{sat})$. We have using \cite[Lemma~5.3]{raicu-regularity}

\begin{equation}\label{eq:def-Z_p^d}
\Z(\X_p^{d}) = \left\{(\z,l): 
\begin{aligned} 
& 0\leq l\leq p-1,\ \z\in\P_n,\ z_1=\cdots=z_{l+1}\leq d-1, \\  
& |\z|+(d-z_1)\cdot l  + 1\leq p\cdot d \leq |\ul{z}|+(d-z_1)\cdot (l+1)
\end{aligned}
\right\}
\end{equation}
which in turn based on (\ref{eq:Z-saturation}) implies
\begin{equation}\label{eq:def-Z_p^d-sat}
\Z((\X_p^d)^{sat}) = \{(\z,l)\in\Z(\X_p^{d}): l \geq 1\}
\end{equation}
and
\begin{equation}\label{eq:def-Z_p^(d)}
\begin{aligned}
\Z(\X_p^{(d)}) &= \{(\z,l)\in\Z(\X_p^{d}): l \geq p-1\} \\
&= \{(\z,p-1): \z\in\P_n, z_1=\cdots=z_p, z_p + z_{p+1} + \cdots + z_n \leq d-1\}\\
\end{aligned}
\end{equation}

\begin{example}\label{example:ZX_3^3}
 We continue with the situation from Example~\ref{example:X_3^3}: $m=n=4$ and $p=d=3$. One can check either directly from Definition~\ref{def:ZX}, or based on (\ref{eq:def-Z_p^d}--\ref{eq:def-Z_p^(d)}) that (if we write $\emptyset$ for the Young diagram of the $(0)$ partition)
 \[\Z(\X_3^3) = \left\{ \left(\ \ydiagram{2,2,2}\ ,2\right), \left(\ \ydiagram{1,1,1}\ ,2\right), \left(\ \ydiagram{1,1,1,1}\ ,2\right), \left(\emptyset,2\right), \left(\ \ydiagram{2,2,2,1}\ ,1\right), \left(\ \ydiagram{2,2,2,2}\ ,0\right)  \right\}\]
 \[\Z((\X_3^3)^{sat}) = \left\{ \left(\ \ydiagram{2,2,2}\ ,2\right), \left(\ \ydiagram{1,1,1}\ ,2\right), \left(\ \ydiagram{1,1,1,1}\ ,2\right), \left(\emptyset,2\right), \left(\ \ydiagram{2,2,2,1}\ ,1\right) \right\}\]
 \[\Z(\X_3^{(3)}) = \left\{ \left(\ \ydiagram{2,2,2}\ ,2\right), \left(\ \ydiagram{1,1,1}\ ,2\right), \left(\ \ydiagram{1,1,1,1}\ ,2\right), \left(\emptyset,2\right) \right\}\]
\end{example}

It is worthwhile to observe that the sets $\Z(\X_p^{(d)})$ in (\ref{eq:def-Z_p^(d)}) get larger as $d$ grows, which in view of (\ref{eq:MIX}) and Theorem~\ref{thm:main-Ext} implies that for every $d\geq 1$ and every $j\geq 0$ the induced maps
\begin{equation}\label{eq:injExt-symb}
 \Ext^j_S(S/I_p^{(d-1)},S) \lra \Ext^j_S(S/I_p^{(d)},S)
\end{equation}
are injective. This is certainly not the case if we replace symbolic powers with the usual powers, as seen for instance in the next example.

\begin{example}\label{ex:Z-2x2minors}
 Assume that we are in the situation from Examples~\ref{example:X-for-powers} and~\ref{example:X-for-symbsat-powers}: $m=n=3$ and $p=2$. Since $\Z(\X_p^{(d-1)}) \subseteq \Z(\X_p^{(d)})$ and $ \Z(\X_p^{(d)}) \subseteq  \Z(\X_p^d)$, we will record for the sake of compactness only the difference between these sets in the following table.
\begin{center}
\setlength{\extrarowheight}{2pt}
\ytableausetup{smalltableaux,aligntableaux=center}
\tabulinesep=1.2mm
\begin{tabu}{c|c|c}
$d$ & $\Z(\X_2^{(d)})\setminus\Z(\X_2^{(d-1)})$ & $\Z(\X_2^d) \setminus \Z(\X_2^{(d)})$ \\
\hline
1 & ($\emptyset$,1) & -- \\
\hline
2 & $\left(\ \ydiagram{1,1}\ ,1\right)$ & $\left(\ \ydiagram{1,1,1}\ ,0\right)$ \\
\hline
3 & $\left(\ \ydiagram{2,2}\ ,1\right) \quad \left(\ \ydiagram{1,1,1}\ ,1\right)$ & $\left(\ \ydiagram{2,2,1}\ ,0\right)$ \\
\hline
4 & $\left(\ \ydiagram{3,3}\ ,1\right) \quad \left(\ \ydiagram{2,2,1}\ ,1\right)$ & $\left(\ \ydiagram{2,2,2}\ ,0\right) \quad \left(\ \ydiagram{3,2,2}\ ,0\right) \quad \left(\ \ydiagram{3,3,1}\ ,0\right)$ \\
\hline
5 & $\left(\ \ydiagram{4,4}\ ,1\right) \ \left(\ \ydiagram{3,3,1}\ ,1\right) \ \left(\ \ydiagram{2,2,2}\ ,1\right)$  & $\left(\ \ydiagram{3,3,2}\ ,0\right) \ \left(\ \ydiagram{3,3,3}\ ,0\right) \ \left(\ \ydiagram{4,4,1}\ ,0\right) \ \left(\ \ydiagram{4,3,2}\ ,0\right)$ \\
\end{tabu}
\end{center}
\end{example}

The last piece of mystery in Theorem~\ref{thm:main-Ext} is the explicit calculation of $\Ext^{\bullet}_S(M,S)$ when $M=J_{\z,l}$. This is the content of the following (slightly weaker) version of \cite[Thm~2.5]{raicu-regularity} and \cite[Thm~3.3]{raicu-weyman}.

\begin{theorem}\label{thm:ExtJzl}
 Fix an integer $0\leq l< n$ and assume that $\z\in\P_n$ is a partition with $z_1=z_2=\cdots=z_l=z_{l+1}$. For $0\leq s\leq t_1\leq\cdots\leq t_{n-l}\leq l$ we consider the set $W(\z,l;\t,s)$ of dominant weights $\ll\in\bb{Z}^n$ satisfying
 \begin{equation}\label{eq:lam-in-W}
 \begin{cases}
 \ll_n = l - z_l - m, & \\
 \ll_{t_i+i} = t_i - z_{n+1-i} - m \quad\rm{for}\quad i=1,\cdots,n-l, & \\
 \ll_s \geq s-n \quad\rm{and}\quad \ll_{s+1} \leq s-m. & \\
 \end{cases}
 \end{equation}
 Letting $\ll(s) = (\ll_1,\cdots,\ll_s,(s-n)^{m-n},\ll_{s+1}+(m-n),\cdots,\ll_n+(m-n))\in\bb{Z}^m$, we~have
 \begin{equation}\label{eq:Extj}
 \Ext^j_S(J_{\z,l},S) = \bigoplus_{\substack{0\leq s\leq t_1\leq\cdots\leq t_{n-l}\leq l \\ m\cdot n - l^2 - s\cdot(m-n) - 2\cdot\left(\sum_{i=1}^{n-l} t_i\right)=j \\ \ll\in W(\z,l;\t,s)}} \bb{S}_{\ll(s)}\bb{C}^m \oo \bb{S}_{\ll}\bb{C}^n,
 \end{equation}
 where $\bb{S}_{\ll(s)}\bb{C}^m \oo \bb{S}_{\ll}\bb{C}^n$ appears in degree $|\ll|=\ll_1 + \cdots + \ll_n$. 
\end{theorem}

The formulas in Theorem~\ref{thm:ExtJzl} are quite involved, but nevertheless they are completely explicit. In particular they allow us to determine which graded components of the modules $\Ext^j_S(J_{\z,l},S)$ are non-zero, and as a consequence determine a formula for the Castelnuovo--Mumford regularity of $J_{\z,l}$. Combining this with Theorem~\ref{thm:main-Ext}, we get formulas for $\reg(S/I)$ for an arbitrary $\GL$-invariant ideal $I$. Unfortunately these are not closed formulas, but rather they involve an often difficult linear integer optimization problem (see \cite[Theorem~2.6]{raicu-regularity} and \cite[Section~4]{raicu-regularity}). Here we will content ourselves with showing that Theorem~\ref{thm:ExtJzl} combined with Theorem~\ref{thm:main-Ext} recovers the description of the $\Ext$ modules in Theorem~\ref{thm:Ext-GLN}. For another example of the concrete calculation of $\Ext$ modules see \cite[Section~7]{raicu-regularity}.

\begin{example}\label{ex:proof-Ext-GLN}
 Assume that $n=1$, $m=N$, and write $X_i = X_{i1}$ so that $S = \bb{C}[X_1,\cdots,X_N]$. Consider $\X = \{ (d) \}$ for some $d\geq 0$, so that $I_{\X} = \mf{m}^d$. We have seen in Example~\ref{ex:MI-n=1} that $\Z(\X) = \{ ((i),0) : i=0,\cdots,d-1 \}$ and in Example~\ref{ex:Jzl-l=0} that $J_{(i),0} = \mf{m}^i / \mf{m}^{i+1}$. Theorem~\ref{thm:main-Ext} implies that
 \begin{equation}\label{eq:ExtS/m^d}
  \Ext^j_S(S/\mf{m}^d,S) \simeq \bigoplus_{i=0}^{d-1} \Ext^j_S(\mf{m}^i/\mf{m}^{i+1},S)\mbox{ for all }j,
 \end{equation}
so it is enough to compute $\Ext^j_S(\mf{m}^i/\mf{m}^{i+1},S)$ based on Theorem~\ref{thm:ExtJzl}. Fix $(\z,l) = ((i),0)$ for some $0\leq i\leq d-1$ and observe that since $l=0$, (\ref{eq:Extj}) forces $s=t_1=0$ and therefore the only potentially non-zero $\Ext$ module occurs for
\[ j = m\cdot n - l^2 - s\cdot(m-n) - 2\cdot\left(\sum_{i=1}^{n-l} t_i\right) = N.\]
Moreover, we get that $W(\z,l;\t,s) = W((i),0;(0),0)$ consist of a single dominant weight $\ll \in \bb{Z}^1$, namely $\ll = (-i-N)$, and for that weight we have
\[ \ll(s) = \ll(0) = (-1^{N-1},-i-N + N - 1) = (-1^{N-1},-i - 1).\]
This shows that
 \begin{equation}\label{eq:m^i/m^i+1}
\Ext^N_S(\mf{m}^i/\mf{m}^{i+1},S) \simeq \bb{S}_{(-1^{N-1},-i - 1)}\bb{C}^N \oo \bb{S}_{-i-N}\bb{C}^1 \simeq \bb{S}_{(-1^{N-1},-i - 1)}\bb{C}^N
 \end{equation}
where the last isomorphism simply disregards the $\GL_1(\bb{C})$-action on the second factor and is only $\GL_N(\bb{C})$-equivariant. Combining (\ref{eq:ExtS/m^d}) with (\ref{eq:m^i/m^i+1}) yields the conclusion of Theorem~\ref{thm:Ext-GLN}.
\end{example}

\section{Local cohomology modules}\label{sec:H_I}

Having computed the $\Ext$ modules in the previous section, as well as the induced maps between them, the description of local cohomology is a consequence of (\ref{eq:loccoh=limExt}). We begin by recalling that the local cohomology modules $H_I^{\bullet}(S)$ depend on $I$ only up to radical. Moreover, any $\GL$-invariant radical ideal $I \subseteq S = \bb{C}[X_{ij}]$ corresponds to a $\GL$-invariant algebraic subset of the space $\bb{C}^{m\times n}$ of $m\times n$ matrices; since any such algebraic set is the set of matrices of rank $<p$ for some value of $p$, it follows that every $\GL$-invariant radical ideal is of the form $I_p$ for some $p$. It is then sufficient to study
\[ H_{I_p}^{j}(S) = \varinjlim_d \Ext^{j}_S(S/I_p^d,S)\mbox{ for }j\geq 0.\]
Based on Theorem~\ref{thm:main-Ext}, it follows that
\begin{equation}\label{eq:HIp-Mp}
 H_{I_p}^{j}(S) = \bigoplus_{M\in \mc{M}_p} \Ext^{j}_S(M,S),
\end{equation}
where $\mc{M}_p \subseteq \bigcup_d \mc{M}(I_p^d)$ is the subset consisting of those $M$ for which $M \in \mc{M}(I_p^d)$ for all $d\gg 0$. We encourage the reader to verify directly that
\[ \mc{M}_p = \bigcup_{d} \mc{M}(I_p^{(d)}) \overset{(\ref{eq:def-Z_p^(d)})}{=} \{J_{\z,p-1} : \z\in\P_n,\mbox{ and }z_1=\cdots=z_p\},\]
which has an alternative explanation as follows. One can tweak the formula (\ref{eq:loccoh=limExt}) and get
\[H_I^j(S) = \varinjlim_d \Ext^j_S(S/J_d,S)\]
where $(J_{d})_d$ is any sequence of ideals which is cofinal with the sequence of powers $(I^d)_{d}$. If we take $J_d = I_p^{(d)}$ and use the injectivity of the maps (\ref{eq:injExt-symb}) we get that
\[ H_{I_p}^{j}(S) = \bigoplus_{M\in \bigcup_{d} \mc{M}(I_p^{(d)})} \Ext^{j}_S(M,S),\]
as desired. In \cite{raicu-weyman} we have used yet another sequence of ideals to compute local cohomology, namely $J_d = I_{d\times p}$, where $\x = d\times p$ denotes the partition whose Young diagram is the $d\times p$ rectangle, i.e. $x_1=\cdots=x_p=d$ and $x_i = 0$ for $i>p$. Coincidentally, the ideals $I_{d\times p}$ are precisely the ones for which we can compute all the syzygy modules, as explained in the next section. Just as for symbolic powers, it is the case that the sets $\mc{M}(I_{d\times p})$ increase with $d$, and their union is precisely the set $\mc{M}_p$.

To give a cleaner description of the modules $H_{I_p}^{\bullet}(S)$ we introduce some notation. We consider the \defi{generalized binomial coefficients}, also known as \defi{q-binomial coefficients} or \defi{Gauss polynomials}, to be the following polynomials in the indeterminate $q$, depending on non-negative integers $a,b$:
\begin{equation}\label{eq:defqbinomial}
{a\choose b}_q=\frac{(1-q^a)(1-q^{a-1})\cdots(1-q^{a-b+1})}{(1-q^b)(1-q^{b-1})\cdots(1-q)}. 
\end{equation}
If we set $q=1$ then we recover the usual binomial coefficients:
\[ {a\choose b}_1 = {a\choose b} = \frac{a!}{b! \cdot (a-b)!}.\]
As another example, if we let $a=4$ and $b=2$ then
\[ {4\choose 2}_q = 1 + q + 2q^2 + q^3 + q^4.\]
In what follows we consider formal linear combinations $\sum_j A^j \cdot q^j$, where $A^j = \bigoplus_{i\in\bb{Z}} A^j_i$ is a graded $\GL$-representation. We will interpret an equality of the form
\[ \sum_j A^j \cdot q^j = \sum_j B^j \cdot q^j \]
to mean that we have graded $\GL$-equivariant isomorphisms $A^j \simeq B^j$ for all $j$, i.e. that $A_i^j \simeq B_i^j$ for every $i$ and $j$. With these conventions, the formula (\ref{eq:HIp-Mp}) can then be written more explicitly as follows.

\begin{theorem}[{\cite[Thm.~6.1]{raicu-weyman}, \cite[Main Theorem(1)]{raicu-weyman-loccoh}}]\label{thm:loccoh}
If we think of $H_{I_p}^j(S)$ as a graded $\GL$-representation (with the natural grading inherited from that on $S$) then we have an equality
\begin{equation}\label{eq:genfun-loccoh}
\sum_{j\geq 0}H_{I_p}^j(S)\cdot q^j = \sum_{s=0}^{p-1} D_s\cdot q^{(n-p+1)^2+(n-s)\cdot(m-n)}\cdot{n-s-1\choose p-1-s}_{q^2}
\end{equation}
where $D_s$ is a graded $\GL$-representation which decomposes as
\[D_s = \bigoplus_{\substack{\ll=(\ll_1\geq\cdots\geq\ll_n)\in\bb{Z}^n \\ \ll_s\geq s-n \\ \ll_{s+1}\leq s-m}} \bb{S}_{\ll(s)}\bb{C}^m\oo \bb{S}_{\ll}\bb{C}^n\]
with $\bb{S}_{\ll(s)}\bb{C}^m \oo \bb{S}_{\ll}\bb{C}^n$ living in degree $|\ll|=\ll_1+\cdots+\ll_n$.
\end{theorem}

The formula (\ref{eq:genfun-loccoh}) is just a shadow of the deeper $\D$-module structure of the local cohomology modules. More precisely, the graded representations $D_s$ in Theorem~\ref{thm:loccoh} are the underlying vector spaces of the simple $\GL$-equivariant $\D$-modules on $\bb{C}^{m\times n}$, where $D_s$ corresponds to the module supported on rank $\leq s$ matrices \cite{raicu-dmods}. The local cohomology modules have finite length when regarded as $\D$-modules, and the multiplicities of the composition factors $D_s$ in their Jordan--H\"older filtrations can be read off by equating the two sides of (\ref{eq:genfun-loccoh}).

\begin{example}\label{ex:loccoh-n=1}
 Consider $n=1$ (and $m=N$), and take $p=1$ (so that $I_p = \mf{m}$). The equation (\ref{eq:genfun-loccoh}) becomes
 \[\sum_{j\geq 0}H_{\mf{m}}^j(S)\cdot q^j = D_0 \cdot q^N.\]
 This means that $H_{\mf{m}}^j(S) = 0$ for $j<N$, and $H_{\mf{m}}^N(S) \simeq D_0$ as graded representations. Note that the formula for $D_0$ in Theorem~\ref{thm:loccoh} specializes to (\ref{eq:loccoh-GLN}).
\end{example}

\begin{example}\label{ex:loccoh-bigger}
 For a slightly bigger examples which shows that multiplicities bigger than $1$ typically occur in local cohomology modules, consider $m=n=5$ and $p=3$. We get
 \[\sum_{j\geq 0}H_{I_3}^j(S)\cdot q^j = q^9\cdot \left(D_0 \cdot {4\choose 2}_q + D_1 \cdot {3\choose 1}_q + D_2 \cdot {2\choose 0}_q \right)\]
 \[=q^9\cdot(D_0+D_1+D_2) + q^{10}\cdot(D_0+D_1) + q^{11}\cdot(2D_0+D_1) + q^{12}\cdot D_0 + q^{13} \cdot D_0.\]
 This formula shows that $H_{I_3}^j(S) = 0$ for $j<9$ or $j>13$. It also says for instance that $H_{I_3}^9(S) \simeq D_0 \oplus D_1 \oplus D_2$ as graded $\GL$-representations, or that the $\D$-module $H_{I_3}^9(S)$ has length three, with composition factors $D_0,D_1,D_2$, each occurring with multiplicity one. The $\D$-module $H_{I_3}^{11}(S)$ also has length three, but only two distinct composition factors: $D_0$ occurring with multiplicity two, and $D_1$ with multiplicity one. Forgetting the $\D$-module structure we have an isomorphism of graded $\GL$-representations $H_{I_3}^{11}(S) \simeq D_0^{\oplus 2} \oplus D_1$.
\end{example}

We close the section by remarking that analogous versions of Theorem~\ref{thm:loccoh} hold for ideals of minors of a generic symmetric matrix, and ideals of Pfaffians of a generic skew-symmetric matrix (see \cite{raicu-weyman-loccoh} and \cite{raicu-weyman-witt}).

\section{Syzygy modules}\label{sec:Tor}

We keep the notation from the previous sections. Given a $\GL$-invariant ideal $I\subset S=\bb{C}[X_{ij}]$, we are interested in studying the vector spaces of $i$-\defi{syzygies} of $I$ 
\begin{equation}\label{eq:defsyzygies}
B_{i}(I)=\Tor_i^S(I,\bb{C}) 
\end{equation}
which are naturally graded $\GL$-representations. We encode these syzygies into the \defi{equivariant Betti polynomial}
\begin{equation}\label{eq:defequivBetti}
B_{I}(q)=\sum_{i\in\bb{Z}} B_i(I) \cdot q^i,
\end{equation}
where $q$ is an indeterminate, and the coefficients of $B_I(q)$ are graded $\GL$-representations (and they are finite dimensional, unlike in the previous section when we studied local cohomology modules). In the case when $I = I_p$, the problem of describing the syzygies of determinantal ideals was solved in \cite{lascoux} (see also \cite{pragacz-weyman} and \cite[Chapter 6]{weyman}). For the powers of the ideal of maximal minors ($I = I_n^d$) a description of the syzygy modules was given in \cite{akin-buchsbaum-weyman}. Perhaps the most well-known result in this area refers to the case $I = I_n$ (which is a special case of both of the results quoted above), where the syzygies of the ideal of maximal minors are given by the Eagon--Northcott resolution \cite[Appendix~A.2.6]{eisenbud-CA}.

A description of the equivariant Betti polynomials in (\ref{eq:defequivBetti}) is still unknown for an arbitrary $\GL$-invariant ideal $I$. It is known however for an important class ideals which we describe next. Consider positive integers $a,b$ with $a\leq n$ and consider the partition $\x=a\times b$ defined by
\[ x_1=\cdots=x_a = b,\ x_i = 0\mbox{ for }i>a.\]
It is the ideals $I_{a\times b}$ (see (\ref{eq:defIx}) for the general definition of the ideals $I_{\x}$) for which we will be able to describe the syzygy modules in Theorem~\ref{thm:syzygies} below. One quick way to define $I_{a\times b}$ is as the smallest $\GL$-invariant ideal which contains the $b$-th powers of the $a\times a$ minors of the generic matrix $(X_{ij})$. We note that $I_{p\times 1} = I_p$ and that $I_{n\times d} = I_n^d$, so that Theorem~\ref{thm:syzygies} will recover as special cases the aforementioned results of \cites{lascoux,akin-buchsbaum-weyman}. 

Before stating the main result we introduce some notation. If $r,s$ are positive integers, $\a$ is a partition with at most $r$ parts (i.e. $\a_i=0$ for $i>r$) and $\b$ is a partition with parts of size at most $s$ (i.e. $\b_1\leq s$), we construct the partition
\begin{equation}\label{eq:defllrsab}
\ll(r,s;\a,\b)=(s+\a_1,\cdots,s+\a_r,\b_1,\b_2,\cdots). 
\end{equation}
This is easiest to visualize in terms of Young diagrams: one starts with an $r\times s$ rectangle, and attach $\a$ to the right and $\b$ to the bottom of the rectangle. If $r=4$, $s=5$, $\a=(4,2,1)$, $\b=(3,2)$, then
\begin{equation}\label{eq:Yngllrsab}
\ll(r,s;\a,\b)=(9,7,6,5,3,2) \llra \ytableausetup{smalltableaux,aligntableaux=center}
\ytableaushort{\none\none\none\none\none\a\a\a\a,\none\none\none\none\none\a\a,\none\none\none\none\none\a,\none,\b\b\b,\b\b} * {9,7,6,5,3,2}
\end{equation}
Recall that $\mu'$ denotes the conjugate partition to $\mu$ and consider the polynomials $h_{r\times s}(q)$ given by
\begin{equation}\label{eq:defhrs}
h_{r\times s}(q)=\sum_{\a,\b}(\bb{S}_{\ll(r,s;\a,\b)}\bb{C}^m \oo \bb{S}_{\ll(r,s;\b',\a')}\bb{C}^n)\cdot  q^{|\a|+|\b|}, 
\end{equation}
where the sum is taken over partitions $\a,\b$ such that $\a$ is contained in the $\min(r,s)\times(n-r)$ rectangle (i.e. $\a_1\leq n-r$, $\a_1'\leq\min(r,s)$) and $\b$ is contained in the $(m-r)\times\min(r,s)$ rectangle (i.e. $\b_1\leq\min(r,s)$ and $\b_1'\leq m-r$), and the representation $\bb{S}_{\ll(r,s;\a,\b)}\bb{C}^m \oo \bb{S}_{\ll(r,s;\b',\a')}\bb{C}^n$ is placed in degree 
\[|\ll(r,s;\a,\b)| = |\ll(r,s;\b',\a')| = r\cdot s + |\a| + |\b|.\]
With this notation we have the following.

\begin{theorem}[{\cite[Thm.~3.1]{raicu-weyman-syzygies}}]\label{thm:syzygies}
 The equivariant Betti polynomial of the ideal $I_{a\times b}$ is
\[B_{I_{a\times b}}(q)=\sum_{t=0}^{n-a} h_{(a+t)\times(b+t)}(q)\cdot q^{t^2+2t}\cdot{t+\min(a,b)-1\choose t}_{q^2}\]
\end{theorem}

\begin{example}\label{ex:syz-n=1}
 Suppose that $n=1$ (and $m=N$) and let $a=1$, $b=d$ (so that $I_{a\times b} = \mf{m}^d$). We get that
 \[ B_{\mf{m}^d}(q) = h_{1\times d}(q) = \sum_{p=0}^{N-1} (\bb{S}_{d,1^p}\bb{C}^N \oo \bb{S}_{d+p}\bb{C}^1)\cdot q^p\]
 which, if we forget the $\GL_1(\bb{C})$-action on the second factor, recovers the conclusion of Theorem~\ref{thm:syz-GLN}.
\end{example}

\begin{example}\label{ex:syz-22-3x3}
 Suppose that $m=n=3$ and let $a=b=2$. We get that
 \[ B_{I_{2\times 2}}(q) = h_{2\times 2}(q) + h_{3\times 3}(q)\cdot q^3 \cdot {2\choose 1}_{q^2} = h_{2\times 2}(q) + q^3\cdot h_{3\times 3}(q) + q^5\cdot h_{3\times 3}(q).\]
 Before analyzing the terms further, it is useful to record the Betti table of $I_{2\times 2}$ that Macaulay2 \cite{M2} computes (recall the convention that the Betti number $\b_{i,i+j} = \dim_{\bb{C}}\Tor^S_i(I_{2\times 2},S)_{i+j}$ is placed in row $j$, column $i$):
 \begin{equation}\label{eq:betti-table}
 \begin{matrix}
      &0&1&2&3&4&5\\\text{total:}&36&90&84&37&9&1\\\text{4:}&36&90&84&36&9&1\\\text{5:}&\text{.}&\text{.}&\text{.}&\text{.}&\text{.}&\text{.}\\\text{6:}&\text{.}&\text{.}&\text{.}&1&\text{.}&\text      
      {.}\\\end{matrix}
\end{equation}
We claim that $h_{2\times 2}(q)$ corresponds to the strand $[36\quad 90\quad 84\quad 36 \quad 9]$ in the Betti table, $q^3\cdot h_{3\times 3}(q)$ corresponds to the entry $1$ in the Betti table in row $6$ and column $3$, and $q^5\cdot h_{3\times 3}(q)$ corresponds to the entry $1$ in the Betti table in row $4$ and column $5$. Indeed, it follows from (\ref{eq:defhrs}) that
\[ h_{3\times 3}(q) = \bb{S}_{3,3,3}\bb{C}^3 \oo \bb{S}_{3,3,3}\bb{C}^3.\]
Note that $\bb{S}_{3,3,3}\bb{C}^3 \oo \bb{S}_{3,3,3}\bb{C}^3$ is a one-dimensional representation of $\GL$, concentrated in degree $9=3+3+3$. It follows from (\ref{eq:defequivBetti}) that $q^3\cdot h_{3\times 3}(q)$ contributes a one-dimensional subspace to $\Tor^S_3(I_{2\times 2},\bb{C})_9$, so it must be the case that
\[\Tor^S_3(I_{2\times 2},\bb{C})_9 \simeq \bb{S}_{3,3,3}\bb{C}^3 \oo \bb{S}_{3,3,3}\bb{C}^3\]
since the Betti number $\b_{3,9}$ is equal to $1$. Similarly, $q^5\cdot h_{3\times 3}(q)$ contributes to $\Tor^S_5(I_{2\times 2},\bb{C})_9$, so
\[\Tor^S_5(I_{2\times 2},\bb{C})_9 \simeq \bb{S}_{3,3,3}\bb{C}^3 \oo \bb{S}_{3,3,3}\bb{C}^3.\]
The remaining entries of the Betti table are accounted for by (for compactness we write $\bb{S}_{\ll}$ instead of $\bb{S}_{\ll}\bb{C}^3$)
\[
\begin{aligned}
h_{2\times 2}(q) &= \bb{S}_{2,2} \oo \bb{S}_{2,2}  \\
&+ (\bb{S}_{3,2} \oo \bb{S}_{2,2,1} + \bb{S}_{2,2,1} \oo \bb{S}_{3,2}) \cdot q \\
&+ (\bb{S}_{3,3} \oo \bb{S}_{2,2,2} + \bb{S}_{3,2,1} \oo \bb{S}_{3,2,1} + \bb{S}_{2,2,2} \oo \bb{S}_{3,3}) \cdot q^2 \\
&+ (\bb{S}_{3,3,1} \oo \bb{S}_{3,2,2} + \bb{S}_{3,2,2} \oo \bb{S}_{3,3,1}) \cdot q^3 \\
&+ (\bb{S}_{3,3,2} \oo \bb{S}_{3,3,2}) \cdot q^4 
\end{aligned}
\]
We conclude that
\[
\begin{aligned}
\Tor^S_0(I_{2\times 2},\bb{C})_4 &\simeq \bb{S}_{2,2} \oo \bb{S}_{2,2} \\
\Tor^S_1(I_{2\times 2},\bb{C})_5 &\simeq \bb{S}_{3,2} \oo \bb{S}_{2,2,1} \oplus \bb{S}_{2,2,1} \oo \bb{S}_{3,2} \\
\Tor^S_2(I_{2\times 2},\bb{C})_6 &\simeq \bb{S}_{3,3} \oo \bb{S}_{2,2,2} \oplus \bb{S}_{3,2,1} \oo \bb{S}_{3,2,1} \oplus \bb{S}_{2,2,2} \oo \bb{S}_{3,3}\\
\Tor^S_3(I_{2\times 2},\bb{C})_7 &\simeq \bb{S}_{3,3,1} \oo \bb{S}_{3,2,2} \oplus \bb{S}_{3,2,2} \oo \bb{S}_{3,3,1} \\
\Tor^S_4(I_{2\times 2},\bb{C})_8 &\simeq \bb{S}_{3,3,2} \oo \bb{S}_{3,3,2}\\
\end{aligned}
\]
To reconcile this with the Betti numbers in (\ref{eq:betti-table}) it suffices to use the following dimension calculations:
\begin{center}
\setlength{\extrarowheight}{2pt}
\tabulinesep=1.2mm
\begin{tabu}{c|c|c|c|c|c|c|c|c|c}
$\ll$ & $(2,2)$ & $(3,2)$ & $(2,2,1)$ & $(3,3)$ & $(2,2,2)$ & $(3,2,1)$ & $(3,2,2)$ & $(3,3,1)$ & $(3,3,2)$  \\
\hline
$\dim(\bb{S}_{\ll})$ & 6 & 15 & 3 & 10 & 1 & 8 & 3 & 6 & 3\\
\end{tabu}
\end{center}
We have for instance
\[ \b_{2,6}(I_{2\times 2}) = \dim(\Tor^S_2(I_{2\times 2},\bb{C})_6) = 10\cdot 1 + 8\cdot 8 + 1 \cdot 10 = 84.\]
\end{example}

Just as the graded $\GL$-representations in (\ref{eq:genfun-loccoh}) were shadows of a higher structure (they were underlying representations of simple $\GL$-equivariant $\D$-modules), the polynomials $h_{r\times s}(q)$ in (\ref{eq:defhrs}) encode the underlying $\GL$-representations of certain simple modules over the \defi{general linear Lie superalgebra} $\mf{gl}(m|n)$. The key to determining (\ref{eq:defequivBetti}) for arbitrary $\GL$-invariant ideals $I$ is perhaps to get a better grasp on the interplay between the representation theory of $\mf{gl}(m|n)$ and the syzygies of the $\GL$-invariant ideals.

\section{Open questions}\label{sec:open}

As alluded to in the text, Theorem~\ref{thm:main-Ext} in conjunction with Theorem~\ref{thm:ExtJzl} gives rise to formulas (which can be quite involved) for the Castelnuovo--Mumford regularity of arbitrary $\GL$-invariant ideals. If we consider high powers (usual, symbolic, saturated) of determinantal ideals then these formulas simplify quite a bit. We show in \cite{raicu-regularity} that for $d\geq n-1$ we have
\[\reg(I_p^d) = \reg((I_p^d)^{sat}) = p\cdot d + \left\lfloor\frac{p-1}{2}\right\rfloor \cdot \left\lceil\frac{p-1}{2}\right\rceil,\mbox{ and }\reg(I_p^{(d)}) = p\cdot d.\]
As far as low powers are concerned, we obtain a closed formula only when $p=2$ (besides the cases when $p=1$ and $p=n$ which are easy):
\[\reg(I_2^d) =  \reg((I_2^d)^{sat}) = \reg(I_2^{(d)}) = d + n - 1 \mbox{ for }d=1,\cdots,n-1.\]

\begin{problem}
 Give a closed formula for the regularity of small powers (usual, symbolic, saturated) of determinantal ideals when $2<p<n$ and $1<d<n-1$.
\end{problem}

The local cohomology modules $H_{I_p}^{\bullet}(S)$ are finite lenght $\D$-modules, and we explained how formula (\ref{eq:genfun-loccoh}) encodes their composition factors in a Jordan--H\"older filtration. It would be interesting to understand better how these composition factors fit together to form $H_{I_p}^{\bullet}(S)$. 

\begin{problem}
 Describe the extension data required to build up the local cohomology modules $H_{I_p}^{\bullet}(S)$ from their composition factors.
\end{problem}

The $\D$-modules $H_{I_p}^{\bullet}(S)$ are not only of finite length, but they are $\GL$-equivariant, regular, and holonomic. We have learnt from private communication with Andr\'as L\H{o}rincz that in the case when $m>n$ the category of such modules is semisimple, and thus Theorem~\ref{thm:loccoh} yields the decomposition of the modules $H_{I_p}^{\bullet}(S)$ as a direct sum of simples. This conclusion however fails for square matrices (when $m=n$), as can be seen for instance by localizing $S$ at the determinant of the generic $n\times n$ matrix. Nevertheless, in this case the category of regular holonomic $\D$-modules has been given a quiver-type description in \cite{braden-grinberg}, so one would have to identify the subcategory of $\GL$-equivariant modules and within that to locate the modules $H_{I_p}^{\bullet}(S)$.

Given the calculation of $\Ext$ modules for $\GL$-invariant ideals, we can determine the regularity and projective dimension of such ideals, so we have a first approximation of the shape of their minimal free resolution. It would be interesting to carry this further and describe the complete Betti tables.

\begin{problem}
 Determine the syzygies $\Tor^S_{\bullet}(I,\bb{C})$ of an arbitrary $\GL$-invariant ideal $I\subset S = \bb{C}[X_{ij}]$.
\end{problem}

Following the ideas of \cite{deconcini-eisenbud-procesi}, the classification Problem~\ref{prob:classification} was solved for symmetric matrices in \cite{abeasis} and for skew-symmetric matrices in \cite{abeasis-delfra}. It is then interesting to study the homological invariants associated to invariant ideals in these cases.

\begin{problem}
 Solve Problem~\ref{prob:invariants} for the natural action of $G=\GL_n(\bb{C})$ on the ring of polynomial functions on $n\times n$ symmetric (resp. skew-symmetric) matrices.
\end{problem}

In the case of skew-symmetric matrices, the problem of computing $\Ext$ modules has been recently resolved by Michael Perlman \cite{perlman}. The local cohomology modules have been computed for both symmetric and skew-symmetric matrices in \cite{raicu-weyman-loccoh}. As far as syzygy modules are concerned there is only little progress, but just as in the case of general $m\times n$ matrices it is expected that the structure of syzygy modules is controlled by the representation theory of certain Lie superalgebras, specifically the periplectic superalgebras \cite{sam}.

Much of the theory of determinantal rings can be developed over fields of arbitrary characteristic, or over more general base rings, as it is done for instance in the monograph \cite{bruns-vetter}. In contrast, the methods employed in proving the results surveyed in this article are heavily dependent on the characteristic of the field being equal to $0$, and in fact the structure of the basic homological invariants can change drastically from characteristic zero to positive characteristic. For instance, since the rings $S/I_p$ are Cohen-Macaulay (see \cites{hochster-eagon,hodge} or \cite[Section~12.C]{bruns-vetter}) it follows from \cite[Prop.~III.4.1]{peskine-szpiro} that in positive characteristic the local cohomology groups $H_{I_p}^{\bullet}(S)$ are non-zero in only one cohomological degree, which is in stark contrast with the conclusion of Theorem~\ref{thm:loccoh}. Nevertheless, it is natural to consider the following.

\begin{problem}
 Solve Problems~\ref{prob:classification} and~\ref{prob:invariants} for matrix spaces over a field of positive characteristic.
\end{problem}

\section*{Acknowledgements}
The author acknowledges the support of the Alfred P. Sloan Foundation, and of the National Science Foundation Grant No.~1600765.

\end{document}